\documentclass[a4paper,reqno]{amsart}

\PII{}
\makeatletter
\def\@setcopyright{\@empty}
\makeatother

\newtheorem{thm}{Theorem}
\newtheorem{lmm}{Lemma}

\theoremstyle{remark}
\newtheorem{rem}{Remark}
\newtheorem{exm}{Example}

\begin{document}

\title{On some transforms of trigonometric series}

\author{F.~M.\ Berisha}
\address{F.~M.\ Berisha\\
	Faculty of Mathematics and Sciences\\
	University of Prishtina\\
	N\"ena Terez\"e~5\\
	10000 Prishtin\"e\\
	Kosovo}
	\email{faton.berisha@uni-pr.edu}

\author{M.~H.\ Filipovi\'c}

\commby{Gradimir Milovanovi\'c}

\keywords{Accelerating convergence of series,
	Euler transform}
\subjclass{Primary 65B10, Secondary 42A32}
\date{}

\begin{abstract}
We give a transform of convergent trigonometric series
into  equivalent convergent series and sufficient
conditions for the transformed series to converge faster
than the original one.
\end{abstract}

\maketitle

\section{Introduction}

Let
\begin{equation}\label{eq:series-cos}
\sum_{n=1}^\infty a_n\cos (\alpha n+\beta)x
\end{equation}
be a convergent real or complex trigonometric series.
A method of accelerating the convergence
of~\eqref{eq:series-cos} is given in~\cite{sorokin:izv-84}.
It consists of as follows:

Let $r\ne 1$ be a real or complex number and $\Delta_r^k$
a linear operator defined by

\begin{align*}
\Delta_r(a_n)       &=a_{n+1}-ra_n\\
\Delta_r^{k+1}(a_n) &=\Delta_r(\Delta_r^k(a_n))
  \quad (k=1,2,\dotsc).
\end{align*}
If $\lim_{n\to \infty}\frac{a_{n+1}}{a_n}=r$,
$\lim_{n\to\infty}
  \frac{\Delta_r^k(a_{n+1})}{\Delta_r^k(a_n)}=r$
$(k=1,2,\dots,p)$,
then
\begin{multline}\label{eq:rem-mod-cos-a}
\sum_{n=1}^\infty a_n\cos(\alpha n+\beta)x
  =\frac{a_1C_r^1(0)}{1-2r\cos\alpha x+r^2}
		+\sum_{k=1}^{p-1}
		  \frac{\Delta_r^k(a_1)C_r^{k+1}(0)}
		    {(1-2r\cos\alpha x+r^2)^{k+1}}\\
+\frac1{(1-2r\cos\alpha x +r^2)^p}
		\sum_{n=1}^\infty\Delta_r^p(a_n)
		  \Delta_r^p\cos(\alpha n+\beta)x,
\end{multline}
where are
$C_r^k(n)=\Delta_r^k\cos(\alpha n+\beta)x$
$(n=0,1,2,\dotsc)$.

A generalisation for number series is given
in~\cite{milovanovic-k-c:izv-88} and for  power series
in~\cite{milovanovic-k-c-k:j-natur-89}.
More detailed aproach on these issues is given
in~\cite{filipovic:master-96}.
In this paper we obtain a generalisation  of
transform~\eqref{eq:rem-mod-cos-a} for cosine
and sine series and  give sufficient conditions for
the modified transform to converge faster
than~\eqref{eq:series-cos}.

For a sequence of real or complex numbers
$\{a_n\}_{n=1}^\infty$
and a given sequence $\{r_n\}_{n=1}^\infty$ we define
a linear operator $L_{r_1\dots r_p}$ by
\begin{align}\label{eq:L}
L_{r_1}(a_n)
  &=a_{n+1}-r_1a_n\notag\\
L_{r_1\dots r_{p+1}}(a_n)
  &=L_{r_1\dots r_p}(a_{n+1})-r_{p+1}L_{r_1\dots r_p}(a_n)
	\quad (p=1,2,\dotsc).
\end{align}
In particular, for the sequence
$\{\cos(\alpha n+\beta)x\}_{n=1}^\infty$ we put
\begin{displaymath}
C_{r_1\dots r_p}(n)
=L_{r_1\dots r_p}\cos(\alpha n+\beta)x
\quad (n=0,1,2,\dotsc).
\end{displaymath}

For fixed~$p$ put
\begin{gather*}
E_0=1, \quad E_1=\sum_{i=1}^p r_i,
  \quad E_2=\sum_{1\le i<j\le p}r_i r_j,\\
E_3=\sum_{1\le i<j<k\le p}r_i r_jr_k,
  \quad \dots,\quad E_p=r_1 r_2\dots r_p
\end{gather*}
(where the summation for~$E_m$ is performed over
all combinations of distinct indices between~$1$ and~$p$
taken~$m$ at a time); we note that
\begin{equation}\label{eq:Drj-a}
L_{r_1\dots r_p}(a_n)=\sum_{k=0}^p(-1)^k E_k a_{n+p-k}.
\end{equation}

In order to establish the modified transform,
we  use the following

\begin{lmm}
Suppose that the coefficients $t_{nm}\quad (0\le m\le n)$
of the infinite triangular matrix $(t_{nm})$ satisfy
the following conditions:
\begin{enumerate}
\item\label{it:toeplitz-1} $\lim_{n\to\infty}t_{nm}=0$
  for each fixed $m$;
\item\label{it:toeplitz-2} there exists a constant~$K$
  such that $\sum_{k=0}^p|t_{pk}|\le K$ for each
  nonnegative~$p$.
\end{enumerate}
Let $\{x_n\}_{n=1}^\infty$ be a sequence and define
the sequence $\{x_n'\}_{n=1}^\infty$ by
\begin{displaymath}
x_n'=\sum_{k=0}^n t_{nk}x_k \quad (n=0,1,2,\dotsc).
\end{displaymath}
Then we have: if $\lim_{n\to\infty}x_n=0$,
then $\lim_{n\to\infty}x_n'=0$.
\end{lmm}

The proof of the lemma is due to
Toeplitz \cite[p. 325]{fikhtengolts:kurs-1}.

\section{A modified transform of trigonometric series}

The following theorem gives a generalisation of
transform~\eqref{eq:rem-mod-cos-a}
for trigonometric series~\eqref{eq:series-cos}.

\begin{thm}\label{th:cr:rem-gen-cos}
Let~\eqref{eq:series-cos} be a convergent real or complex
trigonometric series $(\alpha\ne0)$,
$r_1,\dots,r_p$ $(r_j e^{\pm\alpha xi}\ne1,j=1,\dots,p)$
arbitrary real or complex numbers. Then
\begin{multline}\label{eq:rem-gen-cos-a}
\sum_{n=1}^\infty a_n\cos(\alpha n+\beta)x
  =\frac{a_1C_{r_1}(0)}{1-2r_1\cos \alpha x+r_1^2}\\
	  +\sum_{k=1}^{p-1}
	    \frac{L_{r_1...r_k}(a_1)C_{r_1...r_{k+1}}(0)}
		    {(1-2r_1\cos\alpha x+r_1^2)
		      \dots(1-2r_{k+1}\cos\alpha x+r_{k+1}^2)}\\
+\bigg(\prod_{j=1}^p(1-2r_j\cos\alpha x+r_j^2)\bigg)^{-1}
      \sum_{n=1}^\infty L_{r_1\dots r_p}(a_n)
		    L_{r_1\dots r_p}\cos(\alpha n+\beta)x.
\end{multline}
\end{thm}

\begin{proof}
Considering the Euler's formula for
$\cos(\alpha n+\beta)x$ we have
\begin{displaymath}
\sum_{n=1}^\infty a_n\cos(\alpha n+\beta)x
	=\frac12\sum_{n=1}^\infty a_ne^{(\alpha n +\beta)xi}
		+\frac12\sum_{n=1}^\infty a_ne^{-(\alpha n+\beta)xi}.
\end{displaymath}
Let $f_1(x)=\sum_{n=1}^\infty a_ne^{(\alpha n +\beta)xi}$,
$f_2(x)=\sum_{n=1}^\infty a_ne^{-(\alpha n+\beta)xi}$.
Then
\begin{displaymath}
f_1(x)=a_1e^{(\alpha+\beta)xi}+e^{\alpha xi}\Phi_1(x),
\end{displaymath}
where
$\Phi_1(x)=\sum_{k=1}^\infty a_{k+1}e^{(\alpha k+\beta)xi}$.
So
\begin{displaymath}
(1-r_1e^{\alpha xi})\Phi_1(x)
=\sum_{k=1}^\infty(a_{k+1}-r_1a_k)e^{(\alpha k+\beta)xi}
	+r_1a_1e^{(\alpha+\beta)xi}.
\end{displaymath}
Thus
\begin{displaymath}
\Phi_1(x)
=\frac{r_1a_1e^{(\alpha +\beta)xi}}{1-r_1e^{\alpha xi}}
	+\frac1{1-r_1e^{\alpha xi}}
	  \sum_{k=1}^\infty(a_{k+1}-r_1a_k)e^{(\alpha k+\beta)xi}
\end{displaymath}
and hence
\begin{multline}\label{eq:gen-cos-S-1}
f_1(x)=a_1e^{(\alpha+\beta)xi}
		+e^{\alpha xi}
		  \frac{r_1a_1e^{(\alpha+\beta)xi}}{1-r_1e^{\alpha xi}}
		+\frac{e^{\alpha xi}}{1-r_1e^{\alpha xi}}
			\sum_{k=1}^\infty(a_{k+1}-r_1a_k)e^{(\alpha k+\beta)xi}\\
=\frac{a_1e^{(\alpha +\beta)xi}}{1-r_1e^{\alpha xi}}
		+\frac{e^{\alpha xi}}{1-r_1e^{\alpha xi}}
      \sum_{k=1}^\infty L_{r_1}(a_k)e^{(\alpha k +\beta)xi}.
\end{multline}
Applying a similar technique to
$\sum_{k=1}^\infty L_{r_1}(a_k)e^{(\alpha k+\beta)xi}$,
we get
\begin{displaymath}
\sum_{k=1}^\infty L_{r_1}(a_k)e^{(\alpha k+\beta)xi}
=L_{r_1}(a_1)e^{(\alpha+\beta)xi}+e^{\alpha xi}\Phi_2(x)
\end{displaymath}
(where now
$\Phi_2(x)=\sum_{k=1}^\infty
  L_{r_1}(a_{k+1})e^{(\alpha k+\beta)xi}$),
where from using~\eqref{eq:L}, we get
\begin{multline*}
(1-r_2e^{\alpha xi})\Phi_2(x)
	=\sum_{k=1}^\infty(L_{r_1}(a_{k+1})-r_2L_{r_1}(a_k))
	    e^{(\alpha k+\beta)xi}
		+r_2L_{r_1}(a_1)e^{(\alpha+\beta)xi}\\
=\sum_{k=1}^\infty L_{r_1r_2}(a_k)e^{(\alpha k+\beta)xi}
		+r_2L_{r_1}(a_1)e^{(\alpha+\beta)xi}.
\end{multline*}
Hence
\begin{multline*}
\sum_{k=1}^\infty L_{r_1}(a_k)e^{(\alpha k+\beta)xi}
  =L_{r_1}(a_1)e^{(\alpha+\beta)xi}
		+e^{\alpha xi}
		  \frac{r_2e^{(\alpha+\beta)xi}L_{r_1}(a_1)}
		    {1-r_2e^{\alpha xi}}\\
+\frac{e^{\alpha xi}}{1-r_2e^{\alpha xi}}
		\sum_{k=1}^\infty L_{r_1r_2}(a_k)e^{(\alpha k+\beta)xi}.
\end{multline*}
Thus, using~\eqref{eq:gen-cos-S-1}, we obtain
\begin{multline}\label{eq:gen-cos-S-2}
f_1(x)=\frac{a_1e^{(\alpha+\beta)xi}}{1-r_1e^{\alpha xi}}
		+\frac{L_{r_1}(a_1)e^{(2\alpha+\beta)xi}}
		  {(1-r_1e^{\alpha xi})(1-r_2e^{\alpha xi})}\\
+\frac{e^{2\alpha xi}}
      {(1-r_1e^{\alpha xi})(1-r_2e^{\alpha xi})}
		\sum_{k=1}^\infty L_{r_1r_2}(a_k)e^{(\alpha k+\beta)xi}.
\end{multline}
Repeating this process~$p$ times we find that
\begin{equation}\label{eq:gen-cos-S-p}
f_1(x)=\frac{a_1e^{(\alpha+\beta)xi}}{1-r_1e^{\alpha xi}}
	+\sum_{k=1}^{p-1}L_{r_1\dots r_k}(a_1)
		\frac{e^{(\alpha(k+1)+\beta)xi}}
		  {(1-r_1e^{\alpha xi})\dots(1-r_{k+1}e^{\alpha xi})}
	+R_p^{(1)}(x),
\end{equation}
where
\begin{equation}\label{eq:gen-cos-Rp-a}
R_p^{(1)}(x)
=\frac1{(1-r_1e^{\alpha xi})\dots(1-r_pe^{\alpha xi})}
	\sum_{n=1}^\infty L_{r_1\dots r_p}(a_n)
	  e^{(\alpha(n+p)+\beta)xi}.
\end{equation}
Notice that~\eqref{eq:gen-cos-S-1} and~\eqref{eq:gen-cos-S-2}
are the $p=1$ and $p=2$ cases of~\eqref{eq:gen-cos-S-p},
respectively.
Since $f_2(x)=f_1(-x)$, we have
\begin{equation*}\tag{\ref{eq:gen-cos-S-p}$'$}
f_2(x)=\frac{a_1e^{-(\alpha+\beta)xi}}{1-r_1e^{-\alpha xi}}
  +\sum_{k=1}^{p-1}L_{r_1\dots r_k}(a_1)
    \frac{e^{-(\alpha(k+1)+\beta)xi}}
      {(1-r_1e^{-\alpha xi})\dots(1-r_{k+1}e^{-\alpha xi})}
  +R_p^{(2)}(x),
\end{equation*}
where
\begin{equation*}\tag{\ref{eq:gen-cos-Rp-a}$'$}
R_p^{(2)}(x)
=\frac1{(1-r_1e^{-\alpha xi})\dots(1-r_pe^{-\alpha xi})}
	\sum_{n=1}^\infty L_{r_1\dots r_p}(a_n)
	  e^{-(\alpha(n+p)+\beta)xi}.
\end{equation*}
Multiplying the equations~\eqref{eq:gen-cos-S-p}
and~\thetag{\ref{eq:gen-cos-S-p}$'$} by $\frac12$ and
then summing the two together, using~\eqref{eq:Drj-a}
for the sequence
$\{\cos(\alpha n+\beta)x\}_{n=1}^\infty$, we obtain
\begin{multline*}\tag{\ref{eq:gen-cos-S-p}$''$}
\sum_{n=1}^\infty a_n\cos(\alpha n+\beta)x
  =\frac{a_1C_{r_1}(0)}{1-2r_1\cos\alpha x+r_1^2}\\
+\sum_{k=1}^{p-1}
      \frac{L_{r_1\dots r_k}(a_1)C_{r_1\dots r_{k+1}}(0)}
	      {(1-2r_1\cos\alpha x+r_1^2)
	        \dots(1-2r_{k+1}\cos\alpha x+r_{k+1}^2)}
	  +R_p(x),
\end{multline*}
where
\begin{multline*}\tag{\ref{eq:gen-cos-Rp-a}$''$}
R_p(x)=\frac12\left(R_p^{(1)}(x)+R_p^{(2)}(x)\right)\\
=\bigg(\prod_{j=1}^p(1-2r_j\cos\alpha x+r_j^2)\bigg)^{-1}
		\sum_{n=1}^\infty L_{r_1\dots r_p}(a_n)
		  L_{r_1\dots r_p}\cos(\alpha n+\beta)x.
\end{multline*}
Equalities~\thetag{\ref{eq:gen-cos-S-p}$''$}
and~\thetag{\ref{eq:gen-cos-Rp-a}$''$} complete the proof.
\end{proof}

In completely analogous way we obtain the similar transform
for sine series
\begin{multline}\label{eq:rem-gen-sin}
\sum_{n=1}^\infty a_n\sin(\alpha n+\beta)x
  =\frac{a_1S_{r_1}(0)}{1-2r_1\cos\alpha x+r_1^2}\\
+\sum_{k=1}^{p-1}
      \frac{L_{r_1\dots r_k}(a_1)S_{r_1\dots r_{k+1}}(0)}
		    {(1-2r_1\cos\alpha x+r_1^2)
		      \dots(1-2r_{k+1}\cos\alpha x+r_{k+1}^2)}\\
+\bigg(\prod_{j=1}^p(1-2r_j\cos\alpha x+r_j^2)\bigg)^{-1}
    \sum_{n=1}^\infty L_{r_1\dots r_p}(a_n)
      L_{r_1\dots r_p}\sin(\alpha n+\beta)x,
\end{multline}
where
$S_{r_1\dots r_p}(n)=L_{r_1\dots r_p}\sin(\alpha n+\beta)x$
$(n=0,1,2,\dotsc)$.

\begin{rem}
If $L_{r_1\dots r_p}(a_n)=0$ for some $p\ge1$ and
for~$n$ sufficiently large,
then~\eqref{eq:rem-gen-cos-a} and~\eqref{eq:rem-gen-sin}
transform trigonometric series into finite sums.
\end{rem}

\begin{rem}
In particular, for $r_1=r_2=\dots=r_p=r$ we obtain the
transform~\eqref{eq:rem-mod-cos-a}.
\end{rem}

\section{Accelerating convergence of trigonometric series}

The following theorem gives a transform of convergent
trigonometric series~\eqref{eq:series-cos} into
an equivalent convergent series.

\begin{thm}\label{th:cr:euler-gen-cos}
Let~\eqref{eq:series-cos} be a convergent series on~$x$,
$\pi/2\le|\alpha|x\le3\pi/2$ and $\{r_n\}_{n=1}^\infty$
a sequence of positive real numbers such that for some
real $\lambda>1$
\begin{equation}\label{eq:gen-cos-rn}
r_n=O\left(n^{-\lambda}\right) \quad\text{or}
\quad 1/r_n=O\left(n^{-\lambda}\right).
\end{equation}
Then
\begin{multline}\label{eq:euler-gen-cos-a}
\sum_{n=1}^\infty a_n\cos(\alpha n+\beta)x
  =\frac{a_1C_{r_1}(0)}{1-2r_1\cos\alpha x+r_1^2}\\
+\sum_{k=1}^\infty
      \frac{L_{r_1\dots r_k}(a_1)C_{r_1\dots r_{k+1}}(0)}
		    {(1-2r_1\cos\alpha x+r_1^2)
		      \dots(1-2r_{k+1}\cos\alpha x+r_{k+1}^2)}.
\end{multline}
\end{thm}

\begin{proof}
In order to prove~\eqref{eq:euler-gen-cos-a} we need
to show that in~\thetag{\ref{eq:gen-cos-S-p}$''$},
$\lim_{p\to\infty}R_p(x)=0$.
Let $r_n(x)$ be the remainder of the convergent series
$\sum_{n=1}^\infty a_n e^{(\alpha n+\beta)xi}$.
Then $\lim_{n\to\infty}r_n(x)=0$.
Thus, using~\eqref{eq:Drj-a} we get
\begin{multline*}
R_p^{(1)}(x)
=\bigg(\prod_{j=1}^p(1-r_je^{\alpha xi})\bigg)^{-1}
	\sum_{n=1}^\infty e^{(\alpha(n+p)+\beta)xi}
	  \sum_{k=0}^p(-1)^kE_ka_{n+p-k}\\
=\bigg(\prod_{j=1}^p(1-r_je^{\alpha xi})\bigg)^{-1}
  \sum_{k=0}^p(-1)^{p-k}e^{(p-k)xi}E_{p-k}r_k(x).
\end{multline*}
In the lemma of Introduction we put
\begin{displaymath}
t_{pk}=\bigg(\prod_{j=1}^p(1-r_je^{\alpha xi})\bigg)^{-1}
  (-1)^{p-k}e^{(p-k)xi}E_{p-k}
\quad(0\le k\le p).
\end{displaymath}
We have to show that the conditions of the lemma are
satisfied. For each nonnegative integer~$p$ we have
\begin{multline*}
|t_{pk}|=\bigg|
    \bigg(\prod_{j=1}^p(1-r_je^{\alpha xi})\bigg)^{-1}E_{p-k}
  \bigg|\\
=\bigg|
  \sum_{1\le i_1<\dots<i_{p-k}\le p}
    \bigg(\prod_{s=p-k+1}^p(1-r_{i_s}e^{\alpha xi})\bigg)^{-1}
		\prod_{j=1}^{p-k}\frac{r_{i_j}}{1-r_{i_j}e^{\alpha xi}}
	\bigg|.
\end{multline*}
Without lost in generality we may assume that $\alpha>0$.
Then $\pi/2\le\alpha x\le3\pi/2$,
and since $r_j\ge 0$ $(j=1,\dots,p)$, we have
\begin{equation}\label{eq:gen-cos-rj}
\left|1-r_je^{\alpha xi}\right|>1,
\quad
\left|r_j/
  \left(1-r_je^{\alpha xi}\right)
\right|<1
\quad(j=1,\dots,p).
\end{equation}
Hence
\begin{displaymath}
|t_{pk}|
\le\sum_{1\le i_1<\dots<i_{p-k}\le p}
	\prod_{j=1}^{p-k}
	  \left|\frac{r_{i_j}}{1-r_{i_j}e^{\alpha xi}}\right|
\le\binom p{p-k}M^{p-k}\le p^kM^{p-k},
\end{displaymath}
where
$M=\max_{1\le j\le p}
  \left|r_j/\left(1-r_je^{\alpha xi}\right)\right|
<1$.
Thus $\lim_{p\to\infty}t_{pk}=0$,
so the condition~\ref{it:toeplitz-1} of the lemma holds.
Notice that
\begin{equation}\label{eq:sumtpk}
\sum_{k=0}^p|t_{pk}|
\le\bigg|\prod_{j=1}^p(1-r_je^{\alpha xi})\bigg|^{-1}
	\sum_{k=0}^pE_{p-k}
=\bigg|\prod_{j=1}^p(1-r_je^{\alpha xi})\bigg|^{-1}
  \prod_{k=1}^p(1+r_k).
\end{equation}
From~\eqref{eq:gen-cos-rn} we derive that one of the two series
$\sum_{j=1}^\infty r_j$ or $\sum_{j=1}^\infty1/r_j$ converges,
and hence one of the two infinite products
$\prod_{j=1}^\infty(1+r_j)$ or $\prod_{j=1}^\infty(1+1/r_j)$
converges. Put
\begin{displaymath}
K=
\begin{cases}
\prod_{j=1}^\infty(1+r_j),
  &\text{if $\prod_{j=1}^\infty(1+r_j)$ converges}\\
\prod_{j=1}^\infty(1+1/r_j),
  &\text{if $\prod_{j=1}^\infty(1+1/r_j)$ converges},
\end{cases}
\end{displaymath}
using~\eqref{eq:sumtpk} and~\eqref{eq:gen-cos-rj},
we conclude that the condition~\ref{it:toeplitz-2} of
the lemma is also satisfied.
Thus $\lim_{p\to\infty}R_p^{(1)}(x)=0$.
On the other hand, since
\begin{displaymath}
\left|1-r_je^{\alpha xi}\right|
=\left|1-r_je^{-\alpha xi}\right|
\quad (j=1,2,\dotsc),
\end{displaymath}
the inequalities~\eqref{eq:gen-cos-rj} hold true if we
replace~$x$ by~$-x$.
This however, based on~\eqref{eq:gen-cos-Rp-a}
and~\thetag{\ref{eq:gen-cos-Rp-a}$'$}, means that
$\lim_{p\to\infty}R_p^{(2)}(x)=0$.
Hence, using~\thetag{\ref{eq:gen-cos-Rp-a}$''$},
$\lim_{p\to\infty}R_p(x)=0$.
Now~\eqref{eq:euler-gen-cos-a} follows
from~\thetag{\ref{eq:gen-cos-S-p}$''$}.
\end{proof}

Obviously, the result of Theorem~\ref{th:cr:euler-gen-cos}
can be applied for transform~\eqref{eq:rem-gen-sin}
of sine series.
The obtained transform,  analogous
with~\eqref{eq:euler-gen-cos-a}, is
\begin{multline}\label{eq:euler-gen-sin-a}
\sum_{n=1}^\infty a_n\sin(\alpha n+\beta)x
  =\frac{a_1S_{r_1}(0)}{1-2r_1\cos\alpha x+r_1^2}\\
+\sum_{k=1}^\infty
    \frac{L_{r_1\dots r_k}(a_1)S_{r_1...r_{k+1}}(0)}
		  {(1-2r_1\cos\alpha x+r_1^2)
		    \dots(1-2r_{k+1}\cos\alpha x+r_{k+1}^2)}.
\end{multline}

\begin{rem}\label{rm:gen-cos-condition}
Let $r_j>0$ $(j=1,2,\dotsc)$ and
$\pi/2\le|\alpha|x\le3\pi/2$.
Suppose for $p=1,2,\dotsc$ that
$L_{r_1\dots r_p}(a_n)\ne0$ for~$n$ sufficiently large
and that
$\lim_{n\to\infty}
  \frac{L_{r_1\dots r_p}(a_{n+1})}
    {L_{r_1\dots r_p}(a_n)}$
exists.
Since, according to~\eqref{eq:L},
\begin{equation}\label{eq:gen-num-frac}
\begin{aligned}
\frac{L_{r_1}(a_n)}{a_n}
  &=\frac{a_{n+1}}{a_n}-r_1\\
\frac{L_{r_1\dots r_{p+1}}(a_n)}{L_{r_1\dots r_p}(a_n)}
	&=\frac{L_{r_1\dots r_p}(a_{n+1})}
	    {L_{r_1\dots r_p}(a_n)}
	   -r_{p+1}
	\quad (p=1,2,\dotsc),
\end{aligned}
\end{equation}
if we require the additional condition that
the sequence $\{r_n\}_{n=1}^\infty$ is chosen so that
\begin{displaymath}
r_1=\lim_{n\to\infty}\frac{a_{n+1}}{a_n},
\quad
r_{p+1}=\lim_{n\to\infty}
  \frac{L_{r_1\dots r_p}(a_{n+1})}{L_{r_1\dots r_p}(a_n)}
\quad (p=1,2,\dotsc),
\end{displaymath}
then the sequences~\eqref{eq:gen-num-frac} are null-sequences.
Whence
\begin{displaymath}
\lim_{n\to\infty}\frac{L_{r_1...r_p}(a_n)}{a_n}
=\lim_{n\to\infty}
  \frac{L_{r_1}(a_n)}{a_n}
  \frac{L_{r_1r_2}(a_n)}{L_{r_1}(a_n)}
	\dots\frac{L_{r_1\dots r_p}(a_n)}{L_{r_1\dots r_{p-1}}(a_n)}
=0
\quad (p=1,2,\dotsc).
\end{displaymath}
Therefore, according to~\thetag{\ref{eq:gen-cos-S-p}$''$},
\thetag{\ref{eq:gen-cos-Rp-a}$''$}, \eqref{eq:gen-cos-Rp-a},
\thetag{\ref{eq:gen-cos-Rp-a}$'$} and~\eqref{eq:gen-cos-rj},
we conclude that the series on the right-hand side
of~\eqref{eq:rem-gen-cos-a} and~\eqref{eq:rem-gen-sin}
converge faster than the ones on the left-hand side.
\end{rem}

If we require that the sequence $\{r_n\}_{n=1}^\infty$
from Remark~\ref{rm:gen-cos-condition}
satisfies the condition~\eqref{eq:gen-cos-rn},
then the right-hand sides of~\eqref{eq:euler-gen-cos-a}
and~\eqref{eq:euler-gen-sin-a} converge faster than the
left-hand sides.

\begin{exm}
Let $a_n=1/(a^n+b^n)$ $(0<a<b)$.
Then
\begin{displaymath}
r_1=\lim_{n\to\infty}\frac{a_{n+1}}{a_n}=\frac1b,
\quad
r_{p+1}=\lim_{n\to\infty}
    \frac{L_{r_1\dots r_p}(a_{n+1})}{L_{r_1\dots r_p}(a_n)}
	=\frac{a^p}{b^{p+1}}
	\quad (p=1,2,\dotsc).
\end{displaymath}
Obviously, the sequence $\{r_n\}_{n=1}^\infty$ satisfies
the condition~\eqref{eq:gen-cos-rn} of
Theorem~\ref{th:cr:euler-gen-cos}.
In particular, put $a=2$, $b=3$, $\alpha=1$, $\beta=0$,
$x=3\pi/4$; then in order to calculate
the approximate sum of the number series
$\sum_{n=1}^\infty\frac1{2^n+3^n}\cos\frac{3n\pi}4$
with an error not greater than~$10^{-6}$
we must compute the sum of first~$12$ terms.
Applying the transform~\eqref{eq:rem-gen-cos-a} of
Theorem~\ref{th:cr:rem-gen-cos}, the same accuracy is
obtained by computing the sum of first~$7$ terms for $p=1$,
$4$~terms for $p=2$, and $2$~terms for $p=3$.
\end{exm}

\bibliographystyle{plain}
\bibliography{maths}

\end{document}